\documentclass[ws-ccm.cls]{article}
\usepackage{amssymb, bigints,amsthm,amsmath,enumerate,color,enumerate,mathrsfs}

\newtheorem{Thm}{Theorem}[section]
\newtheorem{Cor}[Thm]{Corollary}
\newtheorem{Lem}[Thm]{Lemma}
\newtheorem{Pro}[Thm]{Proposition}

\theoremstyle{definition}
\newtheorem*{Def}{Definition}

\newtheorem{Exa}[Thm]{Example}

\newcommand{\RN}[1]{%
	\textup{\uppercase\expandafter{\romannumeral#1}}%
}
\date{}

\begin{document}
	
	\centerline {\Large{\bf {Ascent and Descent of Weighted Composition Operators}}} 
	\centerline{\Large{\bf {on Lorentz spaces}}}
	
	\vspace*{2mm}
	\centerline{\bf {Gopal Datt\textsuperscript{1} and  Daljeet Singh Bajaj\textsuperscript{2}}}
	
	\centerline{\textsuperscript{1}Department of Mathematics, PGDAV College, University of Delhi, Delhi-110065, India.}
	
	\centerline{\textsuperscript{2}Department of Mathematics, University of Delhi, Delhi-110007, India.}
	\vspace*{1.2mm}
	\centerline{E-Mail: \textsuperscript{1}gopal.d.sati@gmail.com; \textsuperscript{2} daljeet.math.009@gmail.com}
	
	\centerline{}

\begin{abstract}
	The aim of this article is to detect the ascent and descent of weighted composition operators on Lorentz spaces. We investigate the conditions on the measurable transformation $T$ and the complex-valued measurable function $u$ defined on measure space $(X,\mathcal{A},\mu)$ that cause the weighted composition operators on Lorentz space $L(p,q)$, $1<p\leq \infty$, $1\leq q\leq \infty$ to have finite or infinite ascent (descent). We also give a number of examples in support of our findings. 
	
\end{abstract}


\section{Introduction}
	The theory of ``Lorentz spaces" is presented in the work of G. G. Lorentz in \cite{lor1, lor2}. For a pair of numbers $p$ and $q$, ($1< p< \infty , \ 1\le q < \infty $ or $1< p \le \infty ,  \ q = \infty$), the Lorentz  space is defined as the collection of all complex-valued measurable functions $h$ defined on the $\sigma-$finite complete measure space $(X,\mathcal{A},\mu)$ whose norm $\|h\|_{pq} < \infty$, where the norm is defined as follows
	\begin{align*}
		\|h\|_{pq}~ = ~ \begin{cases}\left(\frac{q}{p}\bigintsss_0^\infty \big(
			t^{1/p} h^{**}(t)\big)^q ~\frac{dt}{t}\right)^{1/q}, &
			\  1< p< \infty , \ 1\le q < \infty  \\
			\underset{t>0}{\sup}\,\,\, t^{1/p} h^{**}(t),& \ 1< p \le \infty ,  \ q = \infty
		\end{cases},
	\end{align*} 
	where $h^{**}$ is the average function of non-increasing rearrangement function $h^*$ given by $h^*(t)=\inf \{ s > 0 : \mu\{x\in X : |h(x)|> s \} \le t\}$, $s,~t\geq 0$. It is represented mathematically as $L(p,q)$. These spaces are generalization of classical Lebesgue spaces. In $1966$, R. Hunt \cite{hunt} discussed various properties and useful tools in support to describe the dual of Lorentz spaces. The duality problem regarding Lorentz spaces was examined by Cwikel \cite{cwikel} and Cwikel and Fefferman \cite{cf1,cf2}. To know more about the Lorentz spaces, we suggest the reader to read \cite{cc1,pkjf} and the references therein.

The study of multiplication and composition operators on different function spaces, like $L^p-$spaces, Lorentz spaces, Orlicz spaces, and many more, is the centre of attraction for several mathematicians. Many interesting mathematical results have been developed and continue to develop in the current period as well. A lot of literature is available about the operators, which are formed by the product of composition and multiplication operators, and these operators are known as weighted composition operators. Such operators are  symbolically represented as $W_{u,T}$, where $u$ is a complex-valued measurable function and $T$ is a measurble transformation defined on $(X,\mathcal{A},\mu)$ respectively. The available theory on weighted composition operators involves their boundedness, compactness, closedness of range and other mathematical results. We refer the reader to read the following references \cite{adv,cc2,gks,kumar,tk} with reference therein.

The terms ascent and descent of linear operators were studied by A.E. Taylor in his book \cite{tay1} and applications of these notions in the development of spectral study of compact operators are described in details. One of the important theorem in his book states that ``If the ascent and descent of a linear operator $L$ defined on a normed space $(Y,\|\cdot\|_Y)$ are finite and equal to $k_1$ and $k_2$ respectively, then $k_1=k_2$ and $X$ can be written as direct sum of null space of $L^{k_1}$ and range space of $L^{k_1}$". To look into the properties, examples and other applications, one can refer to \cite{bg, db, emr}.

The notions of ascent and descent have been examined for composition and weighted composition operators on $l_p$ spaces, $L^p$-spaces, Lorentz spaces, Orlicz spaces, and several more. Motivated by the direction of study, we discussed the properties of ascent and descent of weighted composition operators on Lorentz sequence spaces of functions defined on the measure space $(\mathbb{N},2^{\mathbb{N}}, \mu)$, where $\mu$ denotes the counting measure  and $2^{\mathbb{N}}$ denotes the power set of natural numbers $\mathbb{N}$, in \cite{db}. This paper extends the study and explores it on Lorentz spaces of complex valued functions defined over general ($\sigma$-finite) measure spaces. In section 2, we set up the notations used in the paper and recall the definition and important results that help us in our pursuit. In sections 3 and 4, we discuss main results of the paper which provide classes of functions $u$ and $T$  and compute ascent and descent of the induced weighted composition operators $W_{u,T}$. At the end, we discuss some applications of our findings in order to calculate the ascent and descent of weighted composition operators.

\section{Preliminaries}

In this section, we give necessary definitions, theorems and other related informations required to form the base for this paper. We begin with the definition of distribution function, non-increasing rearrangement and average function which lead to define norm on Lorentz space. 

\noindent Let $(X,\mathcal{A},\mu)$ be a $\sigma$-finite complete measure space and $h:X\mapsto \mathbb{C}$ be a measurable function. By $\mu_h$, we mean the distribution function of $h$ and is given by 
$$\mu_h(s) = \mu\big(\{x\in X:|h(x)|> s\}\big) ~~\text{ where } s\geq 0.$$
With the help of $\mu_h$, we define the non-increasing rearrangement $h^*$ as 
$$h^*(t) ~ =~ \inf\, \{ s > 0 :  \mu_h (s) \le t\},  ~ t > 0.$$
For $t>0$,
$$h^{**}(t) ~= ~ \frac{1}{t}\int_0^t h^*(s)ds.$$

\noindent The set of all (equivalence classes of) measurable functions $h$ having finite norm is known as Lorentz space. It is denoted by $L(p, q)(X,\mathcal{A}, \mu)$ (or in the short form we write $L(p, q)$) where $1<p \le \infty, ~1 \le q \le \infty $. For $h\in L(p,q)$, the norm $\|h\|_{pq}$ is given by 
\begin{align*}
	\|h\|_{pq}~ = ~ \begin{cases}\left(\frac{q}{p}\bigintsss_0^\infty \big(
		t^{1/p} h^{**}(t)\big)^q ~\frac{dt}{t}\right)^{1/q}, &
		\  1< p< \infty , \ 1\le q < \infty  \\
		\underset{t>0}{\sup}\,\,\, t^{1/p} h^{**}(t),& \ 1< p \le \infty ,  \ q = \infty
	\end{cases}
\end{align*}
 The Lorentz space $L(p,q)$ is a Banach space with respect to the norm $\| \cdot\|_{pq}$. This norm is equivalent to quasi-norm given by 
\begin{align*}
	\|h\|_{(pq)}~ = ~ \begin{cases}\left(\frac{q}{p}\bigintsss_0^\infty \big(
		t^{1/p} h^{*}(t)\big)^q ~\frac{dt}{t}\right)^{1/q}, &
		\  1< p< \infty , \ 1\le q < \infty  \\
		\underset{t>0}{\sup}\,\,\, t^{1/p} h^{*}(t),& \ 1< p \le \infty ,  \ q = \infty
	\end{cases}.
\end{align*}
It is easy to note that the followings are equivalent for $h\in L(p,q)$. 
\begin{enumerate}[(i)]
	\item $\mu_h(0)=0$,  
	\item $h=0$ almost everywhere on $X$,
	\item $\|h\|_{pq}=0$,
	\item $\|h\|_{(pq)}=0$. 
\end{enumerate} Further, for $p=q$, Lorentz spaces are equivalent to classical Lebesgue spaces, that is, $L^p$ and $L(p, p)$ are equivalent where $1<p \le \infty $. For the details of Lorentz spaces, the readers are referred to \cite{hunt, lor1, lor2}.

Let $T:X\mapsto X$ be a mapping. For each $k\in\mathbb{N}\cup \{0\}$, we define $T^0(A)=I(A)=A$, $\mu_0(A)=\mu(A)$ and  
$$\mu_k(A):=\mu\circ T^{-k}(A)~=~\mu\circ T^{-(k-1)}(T^{-1}(A))~=~\mu\circ T^{-1}(T^{-(k-1)}(A)) ~~\text{for } A\in\mathcal{A}.$$
A mapping $T: X \mapsto X$ is said to be measurable transformation if $T^{-1}(A)$ is measurable set whenever $A$ is measurable set. In addition, if the condition $\mu(A)=0$ implies $\mu\circ T^{-1}(A)=0$ is satisfied then $T$ is called a non-singular measurable transformation. In this case, the measure  $\mu\circ T^{-1}$ is absolutely continuous with respect to the measure $\mu$ and we denote it as $\mu\circ T^{-1}\ll \mu$. Moreover, for every $k\in\mathbb{N}\cup\{0\}$, $T^k : X\mapsto X$ is a non-singular measurable tansformation. Thus, we have
\begin{align}\label{11}
	\cdots\ll \mu_{k+1}\ll \mu_{k}\ll~\mu_{k-1}\ll \cdots\ll \mu_1\ll\mu_0.
\end{align}
Now, by Radon-Nikodym theorem, we get a locally integrable function $h_{T^k} (\geq 0)$ on $X$ such that
\begin{align}\label{12}
	\mu_k(A) = \int_A h_{T^k}(x)\, d\mu(x), \mbox{ for all } A\in \mathcal{A}. 
\end{align}
Let $T$ be a non-singular measurable transformation and $u$ be a complex-valued measurable function on $X$. We define a mapping $W_{u,T}$ from a Lorentz space $L(p,q)$ into the vector space of all complex-valued measurable functions by 
$$W_{u,T}(f)= u\circ T \cdot f\circ T, ~~f\in L(p,q).$$
It can be easily verified that this mapping is a linear tansformation. If $W_{u,T}$ is bounded with range in $L(p,q)$ then $W_{u,T}$ is known as weighted composition operator on $L(p,q)$. On taking $T=I$ and $u=1$, we get $W_{u,T}=M_u$ and $W_{u,T}=C_T$ which are, respectively, said to be a multiplication and composition operator. In \cite{adv}, the followings are proved:
\begin{enumerate}[(i)]
	\item if $u$ and $h_T$ are essentially bounded (i.e. $u, h_T\in L_\infty(\mu)$) then $W_{u,T}$ is bounded on $L(p,q)$, 
	
	\item if $W_{u,T}$ is bounded on $L(p,q)$ and $T(A_\epsilon)\subseteq A_\epsilon$ for each $\epsilon>0$, where $A_\epsilon=\{x:|u(x)|>\epsilon\}$, then $u\in L_\infty(\mu)$.
\end{enumerate}
 Together $(i)$ and $(ii)$ imply the following.

\begin{Thm}
	Let $u$ be the complex valued measurable function and $T$ be the non-singular and measurable transformation such that $h_T$ is essentially bounded $(h_T\in L_{\infty}(\mu))$ and $T(A_\epsilon)\subseteq A_\epsilon$ for every $\epsilon>0$, where $A_\epsilon=\{x:|u(x)|>\epsilon\}$. Then, a necessary and sufficient condition for $W_{u,T}$ to be bounded on $L(p,q)$ $1<p \leq \infty$, $1\leq q\leq \infty$ is that $u\in L_\infty(\mu)$.
\end{Thm}

Let $Y$ be a normed linear space and $L:Y\mapsto Y$ be a bounded operator. We denote the null space and range space of $L$ by $\mathcal{N}(L)$ and $\mathcal{R}(L)$ respectively. Also, note that, the symbol $\mathfrak{B}(Y)$ denotes the collection of all bounded operators on $Y$. For $L\in \mathfrak{B}(Y)$, we define
\begin{itemize}
	\item the ascent of $L$ as 
			$$\alpha(L):=\inf\{k\in\mathbb{N}\cup\{0\}~|~\mathcal{N}(L^{k+1})=\mathcal{N}(L^{k})\}.$$
	\item the descent of $L$ as 
			$$\beta(L):=\inf\{k\in\mathbb{N}\cup\{0\}~|~\mathcal{R}(L^{k+1})=\mathcal{R}(L^{k})\}.$$
\end{itemize}  
If, in any case, the set is empty then that particular term (ascent or descent) is equal to infinite. Several examples are given by A.E. Taylor in his article \cite{tay2}. 
	
	\section{Multiplication Operators}
	 In this section, we discuss the ascent and descent of multiplication operator $M_u$, induced by essentially bounded complex-valued function $u:X\rightarrow \mathbb{C}$, defined on Lorentz space $L(p,q)$, $1< p\leq \infty$ and $1\leq q\leq \infty$. It is readily verified that the ascent and descent of  $M_u$ is equal to one, if $u=0$ almost everywhere (a.e.). We only left with the case where $u\neq 0$ and the discussion results that if the ascent and descent of $M_u$ are finite then these are either equal to zero or one.
	 
	 \begin{Thm} \label{31} 
	 	Let $u\in L_\infty(\mu)$ be a function in such a way that the induce multiplication operator $M_u$ on $L(p,q)$, $1<p\leq \infty$, $1\leq q \leq \infty$ is well defined. If we assume that $u\neq 0$ on $X$, then ascent of $M_u$ is either equal to zero or one.
	 \end{Thm}
	
	\begin{proof}
		The condition $u\neq 0$ gives rise to two cases. In the first case, we assume that $\mu(\{x\in X~:~u(x)=0\}) = 0$,
		 that is, $u\neq 0$ a.e. In this case, we find that $\mathcal{N}(M_u)$, the null space of $M_u$, is $\{0\}$. Let $f\in L(p,q)$ satisfying $M_u f=0$, which means that $\mu_{M_u f}(0)=0$, hence we have 
	\begin{align*}
			\mu_f(0)&=\mu(\{x\in X~|~|f(x)| >0\})\\
			&=\mu(\{x\in X~|~|u(x)|=0~\&~ |f(x)|> 0\})+\mu(\{x\in X~|~|(M_u f)(x)|> 0\})\\
			&\leq \mu(\{x\in X~|~|u(x)|=0\})+\mu_{M_u f}(0)=0,
	\end{align*}
		which implies $f=0$ almost everywhere on $X$. Thus we have $\mathcal{N}(M_u)=\{0\}$ and hence our claim $\alpha(M_u)=0$ is achieved.\\
		In the second case, we assume that $\mu(\{x\in X~:~u(x)=0\})>0$. As the measure $\mu$ is $\sigma-$finite, we can pick a non-zero finite measurable subset $A$ of $\{x\in X~|~u(x)=0\}$. This provide a non-zero characteristic function $\chi_A$ in $\mathcal{N}(M_u)$. This results that $\alpha(M_u)>0$. Now, let $f\in \mathcal{N}(M_u^2)$. Then $u^2(x) f(x)=0$ a.e. on $X$ which implies that $\mu_{M_u^2 f}(0)=0$. It follows that 
		$$\mu_{M_u f}(0)=\mu(\{x\in X~|~|u(x) f(x)| > 0\})=\mu(\{x\in X~|~|u^2(x) f(x)| > 0\})=\mu_{M_u^2 f}(0)=0.$$ 
		This shows that $f\in \mathcal{N}(M_u)$. Hence, $\mathcal{N}(M_u)=\mathcal{N}(M_u^2)$. This means, $\alpha(M_u)=1$. As a conclusion, we obtain that $\alpha(M_u)$ is either $0$ or $1$.
	\end{proof}

Now, we provide examples which illustrate the use of above theorem.

\begin{Exa}\label{32a} 
	Consider the measure space $\left([0,1],\mathcal{A},\mu\right)$ and define $u:[0,1]\rightarrow \mathbb{C}$ as $u(x)=x$.  Then $M_u$ is a bounded operator on Lorentz space $L(p,q)$, $1<p\leq \infty$, $1\leq q \leq \infty$ and $u\neq 0$ a.e.  Hence by Theorem \ref{31}, $\alpha(M_u)=0$. 
\end{Exa}
\begin{Exa}\label{32b} 
	Consider the measure space $\left([0,1],\mathcal{A},\mu\right)$ and define $u(x)$ as $x$, when $\frac{1}{2}\le x\le 1$ and zero otherwise. Again, $M_u$ is a bounded operator on Lorentz space $L(p,q)$, $1<p\leq \infty$, $1\leq q \leq \infty$ and in this case we can say that $u\neq 0$ (but this time it is not the case $u\neq 0$ a.e.).  Hence by Theorem \ref{31}, $\alpha(M_u)=1$. 
\end{Exa}

In Example \ref{32a}, we see that the ascent of multiplication operator $M_u$ is equal to zero. Now, let us find its descent. Take $f\equiv 1$. Since $f$ lies in $L(p,q)$ and $u\in L_\infty(\mu)$, $g_k:=M_u^k f\in \mathcal{R}(M_u^{k})\subseteq L(p,q)$ for every $k\geq 0$. But, we can note that $g_k\notin \mathcal{R}(M_u^{k+1})$. If it is, then we can find $h_k\in L(p,q)$ in such a way that $g_k=M_u^{k+1}h_k$. This implies that $h_k(x)=1/x$ for every $k\in\mathbb{N}\cup\{0\}$, where $x\in (0,1]$.  This leads to a contradiction as this $h_k$ does not lie in $L(p,q)$, indeed, we have $\|h_k\|_{pq}= \infty$. Hence, we obtain that $g_k\in \mathcal{R}(M_u^k)\setminus \mathcal{R}(M_u^{k+1})$ for every non-negative integer $k$. It follows that $\beta(M_u)=\infty$. 

The conclusion, that is $\beta(M_u)=\infty$, remains the same on replacing the measure space $\left([0,1],\mathcal{A},\mu\right)$ by $\left((0,1),\mathcal{A},\mu\right)$ in Example \ref{32a}. In Example \ref{32a}, we notice that the function $u$ is not bounded away from zero. This gives an idea to think for the function $u$ which is bounded away from zero, and our next result detects the descent of multiplication operator on Lorentz space induced by such functions. Recall that, a complex valued function $u$ defined on $X$ is said to be bounded away from zero if $|u(x)|>\epsilon$ a.e. on $X$ for some $\epsilon>0$.

	\begin{Thm}\label{45}
		If $u$ is such that it is bounded away from zero and inducing $M_u$ on $L(p,q)$, $1<p\leq \infty$, $1\leq q \leq \infty$, then the descent of $M_u$ is equal to zero.
	\end{Thm}
	
	\begin{proof}
		Let $h\in L(p,q)$. Define 
		\begin{align*}
			g'(x)=\left\{\begin{array}{cl}
				\frac{h(x)}{u(x)} & \mbox{ if } |u(x)|>\epsilon\\
				0 & \mbox{ if } |u(x)|\leq\epsilon
			\end{array}\right. 
		\end{align*}
		 Then $g'\in L(p,q)$. Indeed,  $h$ and $u$ are measurable and $|u(x)|>\epsilon$ a.e. on $X$, for some $\epsilon>0$. Thus $g'$ is measurable function and, for every $s\geq 0$, we have 
		\begin{align*}
			\mu_{g'}(s)= \mu(\{x\in X~|~|g'(x)|>s\}) &\leq \mu(\{x\in X~|~|h(x)|>s\epsilon\})=\mu_{h}(s\epsilon).
		\end{align*} 
		Thus, the non-increasing rearrangement and average function of $g'$ satisfy that for $t>0$, 
		$$(g')^*(t)\leq \frac{1}{\epsilon} ~h^*(t) \mbox{~~~~ and ~~~~} (g')^{**}(t)\leq \frac{1}{\epsilon}~ h^{**}(t)$$
		respectively. As a consequence, $\|g'\|_{pq}\leq \frac{1}{\epsilon} \|h\|_{pq}$, $1<p\leq \infty$, $1\leq q\leq \infty$. Since $\mu(\{x\in X~|~M_ug'(x)\neq h(x)\})\leq \mu(\{x\in X~|~|u(x)|\leq\epsilon\})=0$, so $M_ug'=h$ almost everywhere on $X$. As a result, $h\in \mathcal{R}(M_u)$. This shows that $\mathcal{R}(M_u)= L(p,q)= \mathcal{R}(M_u^0)$. Thus, $\beta(M_u)=0$.
	\end{proof}

	\section{Weighted Composition Operators}
	
	Now we attempt to discuss the ascent and descent of weighted composition operators $W_{u,T}$, which in particular, for $T \equiv I $, the identity mapping, become multiplication operators and for $u \equiv 1$ become composition operators. The section is divided into two subsections that examine the ascent and descent of weighted composition operators on Lorentz spaces, respectively. Let us define some useful notations which we use throughout this section. For non-negative integer $k$, define $N_0=X$,~  $N_k=\{x\in X~|~u(T^i(x))\neq 0 \text{ for all }1\leq i\leq k\}$, $X_k=\{x\in X~|~h_{T^k}(x)=0\}$ and $L(p,q)(X_k)=\{h\in L(p,q): h=0 \mbox{ a.e. on } X_k^c\}$, where $X_k^c$ indicates the complement of $X_k$.  
	
	\subsection{Ascent}
	We devote this subsection to trace the ascent of weighted composition operators $W_{u,T}:L(p,q)\mapsto L(p,q)$. We first recall that if  mappings $u$ and $T$ are such that ($u\in L_\infty(\mu)$ and $T$ is non-singular with $h_T\in L_\infty(\mu)$) they induce multiplication operator $M_u$ and composition operator $C_T$ on $L(p,q)$ then they induce the weighted composition operator $W_{u,T}$ on $L(p,q)$ as  $W_{u,T}=C_T M_u$. But,  mappings $u$ and $T$ can induce a weighted composition operator $W_{u,T}$  on $L(p,q)$ without being $M_u$ and $C_T$ on $L(p,q)$ are well defined as one can verify that for $u=0$ almost everywhere (a.e.),  $W_{u,T} (=0)$ is a well defined operator on $L(p,q)$ for each choice of  non-singular measurable transformation $T$, and we can deduce that its ascent is equal to one.	Now onwards, our $u$ is always assumed to be a non-zero function, which we will not express explicitly. In this sequel, we first observe the following.

	\begin{Pro}
		 If $u\in L_\infty(\mu)$ and $W_{u,T}\in\mathfrak{B}(L(p,q))$, $1<p\leq \infty$, $1\leq q\leq \infty$, then $L(p,q)(X_k)\subseteq \mathcal{N}(W_{u,T}^k)$ for each integer $k\geq 0$.
	\end{Pro}

	\begin{proof}
	For $k=0$, the result is straightforward, as $L(p,q)(X_0)=\{0\}$. So, we let $k$ be a natural number and $h$ be a member of $L(p,q)(X_k)$. As $h=0$ a.e. on $X_k^c$, $|u(x)|$ is bounded above by $M^* (=\|u\|_{\infty})$ and $T^i$ is non-singular for each $i\in \mathbb{N}$, we have, for $s\ge 0$,
	
	\begin{align*}
		\mu_{W_{u,T}^k h}(s) &= \mu\left(\left\{x\in X~|~\left|\prod\limits_{i=1}^{k} u(T^i(x)) h(T^k(x))\right|>s\right\}\right)\\
		&\leq\mu\left(\left\{x\in X~|~\left|h(T^k(x))\right|>\frac{s}{(M^*)^k}\right\}\right)+\sum_{i=1}^{k}\mu\left(\left\{x\in X~|~\left|u\circ T^i(x)\right|>M^*\right\}\right)\\
		&= \mu_{k}\left(\left\{x\in X~|~|h(x)|>\frac{s}{(M^*)^k}\right\}\right)+\sum_{i=1}^{k}\mu_i\left(\left\{x\in X~|~\left|u(x)\right|>M^*\right\}\right)\\
		&=\int\limits_{\left\{x\in X~|~|h(x)|>\frac{s}{(M^*)^k}\right\}} h_{T^k}(x)\, d\mu(x) = 0.
	\end{align*}
As a result, $\|W_{u,T}^k h\|_{pq} = 0$, $1<p\leq \infty$, $1\leq q\leq \infty$, that is, $h \in \mathcal{N}(W_{u,T}^k)$.
\end{proof}

The  reverse relation in $L(p,q)(X_k)\subseteq \mathcal{N}(W_{u,T}^k)$ can be obtained when $h_T=0$ a.e. on $(Supp(u))^c$, where the symbol $Supp(u)$ indicates the support of $u$ and $(Supp(u))^c=X\setminus Supp(u)$.

	\begin{Pro}\label{32}
		Let $u:X\mapsto \mathbb{C}$ and $T:X\mapsto X$ be the non-singular mappings such that $W_{u,T}\in \mathfrak{B}(L(p,q))$, $1<p\leq \infty$, $1\leq q\leq \infty$. If $h_T=0$ a.e. on $(Supp(u))^c$ then $\mathcal{N}(W_{u,T}^k)\subseteq L(p,q)(X_k)$ for every non-negative integer $k$.  
	\end{Pro}

	\begin{proof}		Evidently, the result is proved for $k=0$, as $L(p,q)(X_0)=\{0\}=  \mathcal{N}(I)= \mathcal{N}(W_{u,T}^0)$. Let $k\geq 1$ and $h\in \mathcal{N}(W_{u,T}^k)$. Then
\begin{align*}
\mu_{W_{u,T}^k h}(0) = \mu\big(\{x\in X~|~|(W_{u,T}^k h)(x)|>0\}\big)=0.
\end{align*}
Now using the hypothesis $h_T=0$ a.e. on $(Supp(u))^c$ and equation \eqref{11}, we obtain that
	\begin{align*}
	\mu_{h\circ T^k}(0) &=\mu\big(\{x\in X~:~ |h\circ T^k(x)|> 0\}\big)\\
	&= \mu\bigg(\{x\in X:|h\circ T^k(x)|>0\}\cap (\cap_{i=1}^{k}\{x\in X:|u\circ T^i(x)|> 0\})\bigg)\\ 
		&+ \mu\bigg(\{x\in X:|h\circ T^k(x)|> 0\}\cap (\cup_{i=1}^{k}\{x\in X:|u\circ T^i(x)|= 0\})\bigg)\\
		&\leq \mu_{W_{u,T}^k h}(0)+\sum_{i=1}^{k}\mu_i\big(\{x\in X:|u(x)|=0\}\big)\\
		&=0.
	\end{align*}

Thus $h\circ T^k=0$ almost everywhere on $X$. As a consequence, we have 
		\begin{align*}
			0&= \mu_k\left(\{x\in X~|~h(x)\neq 0\}\right)\\
			&= \int\limits_{E} h_{T^k}(x)\, d\mu(x)+\int\limits_{F} h_{T^k}(x)\, d\mu(x)\\
			&\geq \frac{1}{n}\int_{F_n\cap F}d\mu=\frac{1}{n}\mu(F_n\cap F)
		\end{align*}
		for each $n\in \mathbb{N}$, where $E=\{x\in X_k~|~h(x)\neq 0\}$, $F=\{x\in X\setminus X_k~|~h(x)\neq 0\}$ and $F_n=\{x\in X\setminus X_k~|~h_{T^k}(x)>1/n\}$. This indicates that $\mu(F_n\cap F)=0$ for every $n$. As $F=\cup_{n=1}^{\infty}(F_n\cap F)$, we get $\mu(F)=0$. Therefore, $h=0$ a.e. on $X\setminus X_k$. Thus, $\mathcal{N}(W_{u,T}^k)\subseteq L(p,q)(X_k)$.
	\end{proof}

With the help of next example we show that the condition $h_T=0$ a.e. on $(Supp(u))^c$ cannot be relaxed in the Proposition \ref{32}. For $k=0$, we always have $L(p,q)(X_0)=\mathcal{N}(W_{u,T}^0)=\mathcal{N}(I)$ and hence the example provides the mappings $u$ and $T$ to verify that on relaxing the condition that $h_T=0$ a.e. on $(Supp(u))^c$,  we have $\mathcal{N}(W_{u,T}^k)\not\subseteq L(p,q)(X_k)$ for all $k\in\mathbb{N}$.
\begin{Exa}
	Let $(\mathbb{N},2^{\mathbb{N}},\mu)$ be the measure space. Define $T:\mathbb{N}\rightarrow \mathbb{N}$ and $u:\mathbb{N}\rightarrow \mathbb{C}$ as 
	$$T(n)=\left\{\begin{array}{cl}
		1 & \mbox{if } n=1,2\\
		n-1 & \mbox{otherwise}
	\end{array}\right. ~~~\mbox{ and }~~~ u(n)=\left\{\begin{array}{cl}
	0 & \mbox{if } n=1\\
	1 & \mbox{otherwise}
\end{array}\right. .$$ 
By simple calculations we obtain that the Radon-Nikodym derivative $h_{T^k}$ is given by  $$h_{T^k}(n)=\left\{\begin{array}{cl}
	k+1 & \mbox{if } n=1\\
	1 & \mbox{if } n\geq 2
\end{array}.\right.$$ Further, we see that $L(p,q)(X_k)=\{0\}$  and $\chi_{\{1\}}\in \mathcal{N}(W_{u,T}^k)$ for every $k\in\mathbb{N}$. Consequently, $\mathcal{N}(W_{u,T}^k)\not\subseteq L(p,q)(X_k)$ for every natural number $k$.
\end{Exa}
	
	As a result of the previous two propositions, we obtain that the null space of $W_{u,T}^k$ is the collection of all complex valued measurable functions on $X$ having value zero almost everywhere on $X_k^c$.
	
	\begin{Pro}\label{25}
		Assume that a mapping $u\in L_\infty(\mu)$ and a non-singular measurable transformation $T:X\rightarrow X$ induce the weighted composition operator $W_{u,T}$ on Lorentz space $L(p,q)$, $1<p\leq \infty$, $1\leq q\leq \infty$. If $h_T=0$ a.e. on $(Supp(u))^c$ then $L(p,q)(X_k) = \mathcal{N}(W_{u,T}^k)$ for every $k\in\mathbb{N}\cup\{0\}$.   
	\end{Pro}
	
	The measures $\mu$ and $\nu$ on the measureable space $(X,\mathcal{A})$ are said to be equivalent measures if $\mu\ll\nu$ and $\nu\ll\mu$. The forthcoming theorem generates the relationship between the equivalency of null spaces and measures in the following form. 

	\begin{Lem}\label{34}
		Let $W_{u,T}\in \mathfrak{B}(L(p,q))$, $1<p\leq \infty$, $1\leq q\leq \infty$, where $u\in L_\infty(\mu)$ and $h_T=0$ a.e. on $(Supp(u))^c$. A necessary and sufficient condition for $\mathcal{N}(W_{u,T}^k) = \mathcal{N}(W_{u,T}^{k+1})$ is that the  measures $\mu_k$ and $\mu_{k+1}$ are equivalent for every integer $k\geq 0$.
	\end{Lem}
	
	\begin{proof}
		Assume that $\mu_k\ll\mu_{k+1}\ll\mu_k$ for each non-negative integer $k$. Let $h\in \mathcal{N}(W_{u,T}^{k+1})$. Then, by following the same type of calculations as done in Proposition \ref{32}, we get that $h\circ T^{k+1}=0$ almost everywhere on $X$. Hence, by our assumption, $\mu_{k}(\{x\in X~:~h(x)\neq 0\})=0$. This yields that 		
		\begin{align}\label{e3}
		\mu_{W_{u,T}^k h}(0) &= \mu\big(\{x\in X\hspace{-1mm}:|h\circ T^k(x)|> 0\}\cap (\cap _{i=1}^{k}\{x\in X\hspace{-1mm}:|u\circ T^i(x)|> 0\})\big)\nonumber \\
		&\leq \mu_{k}(\{x\in X~:~|h(x)|> 0\}) = 0.
\end{align} 
	This provides that $\|W_{u,T}^k h\|_{pq}=0$. In other words, $h\in\mathcal{N}(W_{u,T}^k)$. Consequently, we get $\mathcal{N}(W_{u,T}^k)=\mathcal{N}(W_{u,T}^{k+1})$ for all $k\in \mathbb{N}\cup\{0\}$. 

For the converse part, we suppose that the condition $\mathcal{N}(W_{u,T}^k)=\mathcal{N}(W_{u,T}^{k+1})$ holds for each non-negative integer $k$. Let $\mu_{k+1}(A) = 0$ for any measurable set $A$. Then on replacing $h$ by $\chi_A$ in equation \eqref{e3}, it provides that $\chi_A\in\mathcal{N}(W_{u,T}^{k+1})$, and thus, $\chi_A\in\mathcal{N}(W_{u,T}^{k})$. Again, by the ideas used in calculations in Proposition \ref{32}, we deduce that $\mu_k(A)=0$. This shows that $\mu_k\ll \mu_{k+1}$. However, $ \mu_{k+1}\ll \mu_k$  is always true. This completes the proof.
	\end{proof}

In Lemma \ref{34}, if the search of non-negative finite integer $k$ is over for which $\mu_k\ll \mu_{k+1}\ll \mu_k$ then smallest such $k$ is equal to the ascent of a bounded weighted composition operator $W_{u,T}$ on $L(p,q)$, $1<p\leq \infty$, $1\leq q\leq \infty$. This is expressed in the following form.

	\begin{Thm} \label{46} 
Suppose that a pair of a non-singular transformation $T:X\mapsto X$ and a mapping $u:X\mapsto \mathbb{C}$ generates a bounded operator $W_{u,T}$ on $L(p,q)$, $1<p\leq \infty$, $1\leq q\leq \infty$. If $u\in L_\infty(\mu)$ and $h_T=0$ a.e. on $(Supp(u))^c$, then $\alpha(W_{u,T})=k$ if and only if $k$ is the least non-negative integer such that $\mu_k$ and $\mu_{k+1}$ are equivalent measures.   
	\end{Thm}

Lemma 2.1 and Theorem 2.2 of \cite{bg} discuss the ascent of composition operators on Lorentz spaces. When we combine these results with the findings of Proposition \ref{25}, Lemma \ref{34} and Theorem \ref{46}, we finally arrive to the following corollaries:

\begin{Cor}
Let $u$ and $T$ be such that  $C_T, W_{u,T}\in \mathfrak{B}(L(p,q))$, $1<p\leq \infty$, $1\leq q\leq \infty$. If $u\in L_{\infty}(\mu)$ and $h_T=0$ a.e. on $(Supp(u))^c$ then $\mathcal{N}(C_T^k)=L(p,q)(X_k)=\mathcal{N}(W_{u,T}^k)$ for all $k\in\mathbb{N}\cup\{0\}$.		
\end{Cor}

\begin{Cor}
	Let $C_T, W_{u,T}\in \mathfrak{B}(L(p,q))$,  where $u\in L_{\infty}(\mu)$ and $h_T=0$ a.e. on $(Supp(u))^c$. Then the following are equivalent conditions:
		\begin{enumerate}[(a)]
			\item $\mu_{k+1}\ll \mu_{k}\ll\mu_{k+1}$
			\item $\mathcal{N}(C_{T}^k)=\mathcal{N}(C_{T}^{k+1})$
			\item $\mathcal{N}(W_{u,T}^k)=\mathcal{N}(W_{u,T}^{k+1})$
		\end{enumerate}
		\end{Cor}

	\begin{Cor}
		Assume that the non-singular measurable transformation $T:X\rightarrow X$ and the mapping $u:X\rightarrow \mathbb{C}$ generates the bounded operators $C_T$ and $W_{u,T}$ on $L(p,q)$, $1<p\leq \infty$, $1\leq q\leq \infty$. Further, let $u\in L_\infty(\mu)$ and $h_T=0$ a.e. on $(Supp(u))^c$. Then the following are equivalent:
		\begin{enumerate}[(a)]
			\item $\alpha(C_T)$, the ascent of $C_T$, is $k$
			\item $\alpha(W_{u,T})$, the ascent of $W_{u,T}$, is $k$
			\item  $k$ is the least non-negative integer such that the measures $\mu_k$ and $\mu_{k+1}$ are equivalent.
		\end{enumerate}    
\end{Cor}

Next, we discuss conditions which assure the equivalency of the measures $\mu_k$ and $\mu_{k+1}$ so that the ascent of weighted composition operators on Lorentz spaces is finite. Now we write $X_k^c\subseteq T^{k+1}(X_k^c)$ in the sense that $T^{k+1}(X_k^c)$ contains almost every element of $X_k^c$, i.e. $\mu(X_k^c\setminus T^{k+1}(X_k^c))=0$, then we have the following.

	\begin{Thm}\label{410}
		Suppose that the pair $(u,T)$ induces an operator $W_{u,T}\in\mathfrak{B}(L(p,q))$, $1<p\leq \infty$, $1\leq q\leq \infty$. If $\mu\circ T(A)\leq \mu(A)$ for each $A\in\mathcal{A}$ and $X_k^c= X\setminus X_k\subseteq T^{k+1}(X_k^c)$ for some $k\in\mathbb{N}\cup \{0\}$ then the measures $\mu_k$ and $\mu_{k+1}$ are equivalent. In fact, $\mu_m\ll\mu_{k}\ll\mu_m$ for all $m\geq k$.
	\end{Thm} 
	
	\begin{proof}
	The goal is to show that $\mu_k\ll\mu_{k+1}$, since $\mu_{k+1}\ll \mu_k$ is evident. Let $A$ be a measurable set satisfying  $\mu_{k+1}(A)=0$. Then $\mu(A_1)=0$, where $A_1=\{x\in A~|~h_{T^{k+1}}(x)>0\}$. Take $A_2=\{x\in A~|~h_{T^k}(x)>0 \text{ and } h_{T^{k+1}}(x)=0\}$. Assume $\mu(A_2)>0$. Then the condition $\sigma-$finite on the measure $\mu$ provides a non-zero finite measurable subset $B$ of $A_2$. By the given hypothesis, we have $B =  T^{k+1}(T^{-(k+1)}(B))$, as $A_2\subseteq X_k^c$, which yields $\mu(B)\leq \mu(T^{-(k+1)}(B))=\int_B h_{T^{k+1}}(x)\,d\mu(x)=0$, a contradiction. Hence, the assumption $\mu(A_2)>0$ is not true. Thus, we obtain that $\mu(A_2)=0$. This further indicates that $\mu_k(A)=0$. Hence, we are done. Further, by applying the technique of principle of mathematical induction, we can easily see that $\mu_m\ll\mu_{k}\ll\mu_m$ for all $m\geq k$.
	\end{proof}

 As the consequence of the Theorem \ref{46} and Theorem \ref{410}, we deduce the following.
 	
	\begin{Thm}\label{411} 
	Let $T:X\mapsto X$ be a non-singular measurable transformation and $u$ be a complex valued measurable function defined on $X$ such that $W_{u,T}\in\mathfrak{B}(L(p,q))$, $1<p\leq \infty$, $1\leq q\leq \infty$ and $\mu\circ T(A)\leq \mu(A)$ for each $A\in\mathcal{A}$. Also, let the Radon-Nikodym derivative $h_T = 0$ a.e. on $(Supp(u))^c$ and $u\in L_\infty(\mu)$. Then, a sufficient condition for $W_{u,T}$ to have ascent a non-negative integer $k$ is that $k$ is the smallest integer such that $X_k^c\subseteq T^{k+1}(X_k^c)$, where $X_k^c= X\setminus X_k$.  
\end{Thm} 

\begin{proof}
	By Theorem \ref{46} and Theorem \ref{410}, we get that the ascent of $W_{u,T}$ is atmost $k$. Now assume that $k$ is the least non-negative integer satisfying $X_k^c\subseteq T^{k+1}(X_k^c)$. To show $\alpha(W_{u,T})=k$, we claim that $k$ is the least non-negative integer such that the measures $\mu_k$ and $\mu_{k+1}$ are equivalent. On contrary, assume that $\mu_p$ and $\mu_{p+1}$ are equivalent measures for some non-negative integer $p$ strictly less than $k$. Using the equation \eqref{11}, we have $\mu_{p}\ll \mu_{m}\ll \mu$ and $\mu_{m}\ll \mu_p\ll \mu$ for every $m \geq p$. Thus, by the chain rule,
	\begin{align*}
		h_{T^p}(x)&= \frac{d\mu_p}{d\mu_{m}}(x)\cdot h_{T^{m}}(x),\\
		h_{T^{m}}(x)&= \frac{d\mu_{m}}{d\mu_{p}}(x)\cdot h_{T^{p}}(x).
	\end{align*}
As a result, we obtain $X_p=X_{m}$ which further implies that $X_p^c=X_{m}^c$. In particular, for $m=k$ and by given hypothesis, we obtain that $X_p^c\subseteq T^{k+1}(X_p^c)$. Our target is to show that the condition $X_p^c\subseteq T^{p+1}(X_p^c)$ holds, which yields the contradiction as $k$ is the least integer satisfying that condition. Suppose that $\mu\big(X_p^c\setminus T^{p+1}(X_p^c)\big)>0$. Since $\mu$ is $\sigma-$finite, there exists a measurable subset $C$ of $X_p^c\setminus T^{p+1}(X_p^c)$ having non-zero finite measure. The structure of $C$ indicates that every element $y\in C$ is of type $T^{k+1}(x)=T^{p+1}(T^{k-p}(x))$ where $T^{k-p}(x)\in X_p$ and $x\in X_p^c$, in other words, $y=T^{p+1}(z)$ for some $z\in X_p$. Take $B=T^{-(p+1)}(C)\cap X_p$. Then $B$ is measurable set with $\mu_p(B)=0$. This provides 
$$\chi_C\circ T^{p+1}(x)=\left\{\begin{array}{cc}
	1 & x\in B\\
	0 & \text{elsewhere}
\end{array}\right.\leq ~\chi_B(x).$$ 
Therefore, $\chi_C\circ T^{p+1}(T^{p+1}(x))\leq \chi_B\circ T^{p+1}(x)$ for every $x\in X$. We know that for each measurable set $A\in\mathcal{A}$
\begin{align*}
	\|\chi_A\|_{pq}=\left\{\begin{array}{cc}
		(p')^{1/q} (\mu(A))^{1/p}, & 1<p<\infty, 1\leq q<\infty,\\
		(\mu(A))^{1/p}, & 1<p\leq \infty, q=\infty,
	\end{array}\right.
\end{align*}
where $1/p+1/p'=1$. Finally, by using the relation $\|\chi_{T^{-(2p+2)}(C)}\|_{pq}\leq \|\chi_{T^{-(p+1)}(B)}\|_{pq}$, we get $\mu_{2p+2}(C)=0$, and hence, $\mu_p(C)=0$. This yields $\mu(C)=0$, a contradiction. This implies that $X_p^c\subseteq T^{p+1}(X_p^c)$. This contradicts the fact that $k$ is the least non-negative integer satisfying $X_k^c\subseteq T^{k+1}(X_k^c)$. This proves that $k$ is the least non-negative integer such that measures $\mu_k$ and $\mu_{k+1}$ are equivalent. This proves that the ascent of $W_{u,T}$ is equal to $k$.
\end{proof}

	 Before we proceed, it is to be noted that the assumption $h_T=0$ a.e. on $(Supp(u))^c$, in the Theorem \ref{411},  can not be dropped. This we can see through the following example.
 	  
 	 \begin{Exa}
 	 	On the measure space $(\mathbb{N},2^\mathbb{N},\mu)$, if we define $T:\mathbb{N}\rightarrow \mathbb{N}$ and $u:\mathbb{N}\rightarrow \mathbb{C}$ as 
 	 	$$T(n)=\left\{\begin{array}{cl}
 	 		1 & \mbox{ if } n=1\\
 	 		2 & \mbox{ if } n=3\\
 	 		n-2& \mbox{ if $n$ is odd and } n\geq 5\\
 	  	 	n+2 & \mbox{ if $n$ is even }
 	 	\end{array}\right.  ~~~~~\mbox{ and }~~~~u(n)=\left\{\begin{array}{cl}
 	 	0 & \mbox{ if $n$ is odd}\\
 	 	1 & \mbox{ if $n$ is even}
  	\end{array}\right.$$
   then all the conditions of Theorem \ref{411} are satisfied except the condition $h_T=0$ a.e. on $(Supp(u))^c$. Under these choices of $u$ and $T$, we obtain that the ascent of $W_{u,T}$ is infinite, as $\chi_{\{1\}}\in \mathcal{N}(W_{u,T})\setminus \{0\}$ and $\chi_{\{2k\}}\in \mathcal{N}(W_{u,T}^{k+1})\setminus \mathcal{N}(W_{u,T}^k)$ for $k\in\mathbb{N}$.
 	 \end{Exa}

	Now, our next result shows the existence of weighted composition operators $W_{u,T}:L(p,q)\mapsto L(p,q)$, $1<p\leq \infty$, $1\leq q\leq \infty$, having ascent infinity.

	\begin{Thm}\label{38}
	Let $W_{u,T}\in \mathfrak{B}(L(p,q))$, $1<p\leq \infty, 1\leq q \leq \infty$. If there is a sequence  $\left<A_k\right>_{k=0}^{\infty}$ of measurable sets in such a way that measure of 
	\begin{enumerate}[(a)]
		\item $A_k$ is non-zero and finite
		\item $N_k\cap T^{-k}(A_k)$ is not equal to zero
		\item $N_{k+1}\cap T^{-(k+1)}(A_k)$ is equal to zero
	\end{enumerate}
	for every $k\geq 0$, then the ascent of $W_{u,T}$ can not be finite. 
	\end{Thm}

\begin{proof}
	Suppose that there exists a sequence $\left<A_k\right>_{k=0}^{\infty}$ of measurable sets satisfying the properties $(a), (b)$ and $(c)$. In order to show that the ascent of $W_{u,T}$ is infinite, we aim that the  characteristic function $\chi_{A_k}$ lies in  $\mathcal{N}(W_{u,T}^{k+1})\setminus \mathcal{N}(W_{u,T}^{k})$ for all $k\geq 0$. For $s\geq 0$
	\begin{align*}
		\mu_{W_{u,T}^{k+1}\chi_{A_k}}(s)&=\mu\left(\left\{x\in X~|~\left|\prod_{i=1}^{k+1}u\circ T^{i}(x)\cdot\chi_{A_k}\circ T^{k+1}(x)\right|>s\right\}\right)\\
		&\leq \mu(N_{k+1}\cap T^{-(k+1)}(A_k))=0.
	\end{align*}
	It follows that non-increasing rearrangement of $W_{u,T}^{k+1}\chi_{A_k}$ is equal to zero, and as a conclusion, $\left\|W_{u,T}^{k+1}\chi_{A_k}\right\|_{pq}=0$. Thus, $\chi_{A_k}\in \mathcal{N}(W_{u,T}^{k+1})$. Finally, the justification of $\chi_{A_k}\not\in \mathcal{N}(W_{u,T}^{k})$ is trival. In fact, if we consider $\chi_{A_k}\in \mathcal{N}(W_{u,T}^{k})$ then it contradicts the property $(b)$,  as $N_k\cap T^{-k}(A_k) = \{x\in X : |W_{u,T}^k\chi_{A_k}(x)|>0\}$. 
\end{proof}

Our next result, the proof of which follows the lines of proof of Theorem 2.10 of \cite{bg} and hence we just state the result. 

\begin{Thm} \label{414}
	Assume that the image of measurable set under the non-singular measurable transformation $T:X\mapsto X$ is also measurable. Then $\alpha(W_{u,T})=\infty$ implies the existence of measurable sets sequence $\left<A_k\right>_{k=0}^{\infty}$ of $X$ in such a way that for all $k\geq 0$ the following properties hold
	\begin{enumerate}[(a)]
		\item $0<\mu(A_k)<\infty$
		\item $A_k\subseteq T^k(B\cap N_k)$ for some $B\in\mathcal{A}$
		\item $A_k\not\in \left\{T^{k+1}(D\cap N_{k+1})~|~D\in\mathcal{A} \text{ and } \mu(D\cap N_{k+1})>0\right\}$
	\end{enumerate}
\end{Thm}
\begin{proof} 
	The proof follows from the proof of [2, Theorem 2.10] on replacing $X_k'$ by $N_k\cap B$, where $B=\{x\in X: h_k\circ T^k(x)\neq 0\}$. 
\end{proof}

	\subsection{Descent}
This subsection deals with results on the descent of the weighted composition operators $W_{u,T}\in\mathfrak{B}(L(p,q))$, $1<p\leq \infty$, $1\leq q\leq \infty$. To develop the theory on descent, we begin by introducing necessary notations, definitions and lemmas. Following this groundwork, we delve into the main results of this section.

\begin{Def}
	A measurable transformation $T:X\rightarrow X$ is known as injective almost everywhere on $X$ if $\mathcal{E}$ is a measurable set with measure equal to zero where $\mathcal{E}:=\{x\in X~|~\exists ~y\in X, y\neq x \text{ and } T(x)=T(y)\}$.
\end{Def}

Let $k\in\mathbb{N}\cup \{0\}$ be a fixed number. Define $\mathcal{E}_k :=\{x\in T^k(N_{k+1})~|~\exists~y\in T^k(N_{k+1}), y\neq x \text{ and } T(y)=T(x)\}$. The sets $\mathcal{E}_k$ satisfy the relationship $\cdots \subseteq \mathcal{E}_{k+1}\subseteq \mathcal{E}_{k}\subseteq \mathcal{E}_{k-1}\subseteq\cdots\subseteq  \mathcal{E}_{1}\subseteq \mathcal{E}_{0}\subseteq \mathcal{E}$. This can be verified using the following lemma.

\begin{Lem}\label{n415}
	For every $k\in \mathbb{N}\cup \{0\}$, $T^{k+1}(N_{k+2})\subseteq T^{k}(N_{k+1})$.
\end{Lem}
\begin{proof}
	Let $x\in T^{k+1}(N_{k+2})$. Then we can select $y_x$ in $N_{k+2}$ satisfying $x=T^{k+1}(y_x)=T^{k}(T(y_x))$. This results that $x\in T^{k}(N_{k+1})$, as $u\circ T^i(T(y_x))=u\circ T^{i+1}(y_x)\neq 0$ for all $1\leq i\leq k+1$ because of $y_x\in N_{k+2}$.
\end{proof}

Corresponding to each $x\in \mathcal{E}_k$, we define $A_x^k:=\{y\in \mathcal{E}_k~|~T(x)=T(y)\}$. In the next lemma, we demonstrate that for distinct $x_1$ and $x_2$ within $\mathcal{E}_k$, the sets $A_{x_1}^k$ and $A_{x_2}^k$ are either disjoint or identical.  

\begin{Lem} \label{n416}
	For distinct $x_1, x_2\in \mathcal{E}_k$, either $A_{x_1}^k\cap A_{x_2}^k=\emptyset$ or $A_{x_1}^k=A_{x_2}^k$.
\end{Lem}
\begin{proof}
	If $A_{x_1}^k\cap A_{x_2}^k=\emptyset$ then we are done. Suppose that $z\in A_{x_1}^k\cap A_{x_2}^k$. This implies $T(z)=T(x_1)=T(x_2)$, and as a result, $x_2\in A_{x_1}^k$ and $x_1\in A_{x_2}^k$. Hence, $A_{x_1}^k=A_{x_2}^k$. Indeed, if $y\in A_{x_1}^k$ then $T(y)=T(x_1)$ which further gives $T(y)=T(x_2)$. Therefore, $y\in A_{x_2}^k$. Similarly, we get the other side of containment.
\end{proof}

By the definition of injective almost everywhere on $X$, it is evident that if the measure of $A_x^k$ is non-zero for some $x\in \mathcal{E}_k$ then the mapping $T: T^k(N_{k+1})\rightarrow T^{k+1}(N_{k+1})~(\subseteq T^k(N_{k+1}))$ is not injective almost everywhere on $T^k(N_{k+1})$. Consequently, $T:X\rightarrow X$ is not injective almost everywhere on $X$. These sets $A_x^k$ play an important role in ensuring the infinity descent of $W_{u,T}$. To accomplish this, we first introduce separable measurable sets.

\begin{Def}
	A measurable set $E$ of $X$ is said to be separable measurable set if we can write $E$ as the union of two disjoint positive measurable subsets of $E$, that is, there are  $E_1\subseteq E$ and $E_2\subseteq E$ such that $\mu(E_1)>0$, $\mu(E_2)>0$, $E_1\cap E_2=\emptyset$ and $E=E_1\cup E_2$.
\end{Def}

\begin{Lem}\label{new417}
	Let $(X,\mathcal{A},\mu)$ be a $\sigma-$finite and complete measure space. Suppose that there exists atleast one separable subset $A_x^k$ of $\mathcal{E}_k$ satisfying $\mu(N_{k+1}\cap T^{-k}(B))>0$ for all positive measurable subset $B$ of $A_x^k$ where $k\in\mathbb{N}\cup \{0\}$. Then $\beta(W_{u,T})=\infty$.
\end{Lem}
\begin{proof}
	Let $k\in\mathbb{N}\cup\{0\}$ be a fixed number. Since $A_x^k$ is separable subset of $\mathcal{E}_k$, we can find two disjoint positive measurable subsets $A_1$ and $A_2$ of $A_x^k$ in such a way that $A_x^k=A_1\cup A_2$. Due to $\sigma-$finite, we can take $B_1~(\subseteq A_1)$ and $B_2~(\subseteq A_2)$ such that $0<\mu(B_1), \mu(B_2)<\infty$, $B_1\cap B_2=\emptyset$ and $T^{-k}(B_1)\cap T^{-k}(B_2)=\emptyset$. Define $f:=\chi_{B_1}-\chi_{B_2}$. Obviously, $f\in L(p,q)$, and as a result, $g:=W_{u,T}^k f\in \mathcal{R}(W_{u,T}^k)$. Now, we assert that $g\not\in \mathcal{R}(W_{u.T}^{k+1})$. On the contrary, we assume that $g=W_{u,T}^{k+1}h$ for some $h\in L(p,q)$, that is, $\mu\big(\{x\in X~|~g(x)\neq W_{u,T}^{k+1}h(x) \}\big)=0$. Since the measure of $N_{k+1}\cap T^{-k}(B_1)$ and  $N_{k+1}\cap T^{-k}(B_2)$ are positive, we can pick $z_1\in N_{k+1}\cap T^{-k}(B_1)$ and $z_2\in N_{k+1}\cap T^{-k}(B_2)$ satisfying $g(z_1)= W_{u,T}^{k+1}h(z_1)$ and $g(z_2)=W_{u,T}^{k+1}h(z_2)$ respectively. This provides $1=u\circ T^{k+1}(z_1)\cdot h\circ T^{k+1}(z_1)$ and $-1=u\circ T^{k+1}(z_2)\cdot h\circ T^{k+1}(z_2)$ which further leads to a contradiction because $T^{k}(z_1), T^{k}(z_2)\in A_x^k$. 
\end{proof}

\begin{Thm}\label{new418}
	Suppose that $(X,\mathcal{A},\mu)$ is a $\sigma-$finite complete measure space and $T:X\rightarrow X$ is a non-singular measurable transformation satisfying the property $\mu\circ T(E)\leq \mu(E)$ for every $E\in\mathcal{A}$. If there is atleast one separable subset $A_x^k$ of $\mathcal{E}_k$ for every $k\in\mathbb{N}\cup\{0\}$ then the descent of $W_{u,T}$ is infinity.
\end{Thm}
\begin{proof}
	In order to proof this we only need to show that  $\mu(N_{k+1}\cap T^{-k}(B))>0$ for all positive measurable subset $B$ of $A_x^k ~(\subseteq T^k(N_{k+1}))$ where $k\in\mathbb{N}\cup \{0\}$. To attain the target we claim that $T^k(N_{k+1}\cap T^{-k}(B))=B$. If $x\in T^k(N_{k+1}\cap T^{-k}(B))$ then $x=T^k(y)$ for some $y\in N_{k+1}\cap T^{-k}(B)$. This implies $x=T^k(y)\in T^k(N_{k+1})\cap B=B$. Conversely, if $x\in B$ then $x=T^k(y)$ for some $y\in N_{k+1}$. This provides $y\in N_{k+1}\cap T^{-k}(B)$, and as a consequence, we have $x=T^k(y)\in T^k(N_{k+1}\cap T^{-k}(B))$. Hence, $T^k(N_{k+1}\cap T^{-k}(B))=B$. To complete the proof, if possible assume that $\mu(N_{k+1}\cap T^{-k}(B))=0$ whenever $\mu(B)>0$, then using the property of $T$, we achieve $\mu(B)=\mu(T^k(N_{k+1}\cap T^{-k}(B)))=0$, a contradiction. Hence, we are done with the claim, and the result follows from Lemma \ref{new417}.
\end{proof}

In the upcoming theorem, we investigate another sufficient condition that helps in attaining $\beta(W_{u,T})=\infty$. The Lemma \ref{n416} divides $\mathcal{E}_k$ as union of disjoint $A_x^k$ where $x$ belongs to $\mathcal{E}_k$. Now, from each disjoint $A_{x}^k$ we pick $x\in A_{x}^k$, and we denote the collection of all such $x$ by $\mathcal{E}_1^k$, and $\mathcal{E}_k\setminus \mathcal{E}_1^k$ by $\mathcal{E}_2^k$.

\begin{Thm}\label{new419}
	Let $T$ be a non-singular measurable transformation on $\sigma-$finite and complete measure space $(X,\mathcal{A},\mu)$ such that $\mu\circ T(A)\leq \mu(A)$ for each $A\in\mathcal{A}$. Assume that there are measurable subsets $A_1^k~ (\subseteq \mathcal{E}_1^k)$ and $A_2^k ~(\subseteq \mathcal{E}_2^k)$ with non-zero finite measure satisfying $T_k(A_1^k)=T_k(A_2^k)$ where $T_k=T : T^k(N_{k+1})\rightarrow T^{k+1}(N_{k+1})$ for each $k\in\mathbb{N}\cup\{0\}$. Then the operator $W_{u,T}$ on $L(p,q)$ has infinity descent.
\end{Thm}

	\begin{proof}
		Let $k\in\mathbb{N}\cup\{0\}$ be a fixed number such that the conditions given in the statement of the theorem hold, that is, there are measurable subsets $A_1^k~ (\subseteq \mathcal{E}_1^k)$ and $A_2^k ~(\subseteq \mathcal{E}_2^k)$ in such a way $0<\mu(A_1^k), \mu(A_2^k)<\infty$ and  $T_k(A_1^k)=T_k(A_2^k)$ where $T_k=T : T^k(N_{k+1})\rightarrow T^{k+1}(N_{k+1})$. Define $f:=\chi_{A_1^k}-\chi_{A_2^k}$ and $g:=W_{u,T}^k f$. Both are non-zero and $g\in \mathcal{R}(W_{u,T}^k)$. Now, we show that $g$ does not belong to $\mathcal{R}(W_{u,T}^{k+1})$. On the contrary, let us assume that $g=W_{u,T}^{k+1}h$ (a.e. on $X$) for some $h\in L(p,q)$. Let us write 
		$$N_{k+1}\cap T^{-k}(A_1^k) = A_{11}^k\cup A_{12}^k \text{~~and ~~} N_{k+1}\cap T^{-k}(A_2^k) = A_{21}^k\cup A_{22}^k$$
		where $A_{i1}^k=\{x\in N_{k+1}\cap T^{-k}(A_i^k)~|~g(x)\neq W_{u,T}^{k+1}h(x)\}$ and $A_{i2}^k=\{x\in N_{k+1}\cap T^{-k}(A_i^k)~|~g(x) = W_{u,T}^{k+1}h(x)\}=\{x\in N_{k+1}\cap T^{-k}(A_i^k)~|~ (-1)^{i-1} = u\circ T^{k+1}(x) h\circ T^{k+1}(x)\}$ for $i=1,2$. It can also be verified, using the idea of the proof of Theorem \ref{new418}, that the measure of $N_{k+1}\cap T^{-k}(A_1^k)$ and $N_{k+1}\cap T^{-k}(A_2^k)$ are positive. Now, our aim is to show that $T^{k+1}(A_{12}^k)$ is a subset of $T^{k+1}(A_{21}^k)$. Let $y=T^{k+1}(x)$ for some $x\in A_{12}^k$. Since $A_{12}^k\subseteq N_{k+1}\cap T^{-k}(A_1^k)$ and $T_k(A_1^k)=T_k(A_2^k)$, 
		$$y = T^{k+1}(x) = T(T^{k}(x)) = T_k(T^k(x)) = T_k(z_0) = T_k(T^k(z)) = T^{k+1}(z)$$
		where $z_0\in A_2^k ~\left(\subseteq T^k(N_{k+1})\right)$ and $z\in N_{k+1}\cap T^{-k}(A_2^k)$. More precisely, this $z$ is a member of $A_{21}^k$, as if $z\in A_{22}^k$ then 
		$$-1 = u\circ T^{k+1}(z) h\circ T^{k+1}(z) = u\circ T^{k+1}(x) h\circ T^{k+1}(x) = 1$$
		which is a contradiction. This gives our desired claim. In the end, we get an contradiction to the assumption that $g=W_{u,T}^{k+1}h$ (a.e. on $X$) for some $h\in L(p,q)$ by using the properties of $T$, that is, the properties provide  
		$$\mu(T^{k+1}(A_{12}^k))\leq \mu(T^{k+1}(A_{21}^k))\leq \mu(A_{21}^k)\leq \mu\bigg(\{x\in X~|~g(x)\neq W_{u,T}^{k+1}h(x)\}\bigg)=0,$$
		and therefore, $\mu(A_{12}^k)\leq \mu\bigg(T^{-(k+1)}(T^{k+1}(A_{12}^k))\bigg)=0$. As a result, $g$ does not belong to  $\mathcal{R}(W_{u,T}^{k+1})$. Hence, the descent of $W_{u,T}$ is infinity.
		\end{proof}

Now we try to investigate the necessary conditions for the descent of $W_{u,T}\in \mathfrak{B}(L(p,q))$ to be infinity. In this sequel, we attain the following.

\begin{Thm}\label{42}
	Let $(X,\mathcal{A},\mu)$ be a $\sigma-$finite complete measure space. Suppose that $W_{u,T}\in \mathfrak{B}(L(p,q))$, $1<p\leq \infty$, $1\leq q\leq \infty$, $u$ is bounded away from zero and $\mu(T(A))\leq \mu(A)$ for all $A\in\mathcal{A}$. Then, a necessary condition for $\beta(W_{u,T})=\infty$ is that the map $T_k = T|_{T^k(X)} : T^k(X)\rightarrow T^{k+1}(X)$ is not injective almost everywhere on $T^k(X)$ for every  $k\in\mathbb{N}\cup \{0\}$.
\end{Thm}

\begin{proof}
	The proof is done by using contrapositive statement, that is, if there exists an integer $k\geq 0$ such that $T_k= T|_{T^k(X)}: T^k(X)\rightarrow T^{k+1}(X)$ is injective a.e. on $T^k(X)$ then the descent of $W_{u,T}$ is finite and less than equal to $k$. Let $k\in\mathbb{N}\cup\{0\}$ be a fixed number for which the map $T_k$ is injective a.e. on $T^k(X)$. So, by definition, the measure of $\mathcal{F}_k$ is zero where 
	$$\mathcal{F}_k:=\{x\in T^k(X)~|~\exists~ y\in T^k(X), y\neq x \text{ and } T_k(x)=T_k(y)\}.$$
	Further, by definition of bounded away from zero, there is $\epsilon>0$ satisfying $|u(x)|>\epsilon$ a.e. on $X$. Let us take 
	$$\mathcal{R}_1:=\{x\in X~|~|u(x)|>\epsilon\} \text{~~and~~}\mathcal{R}_2:=\{x\in X~|~|u(x)|\leq \epsilon\}.$$
	Obviously, $\mu(\mathcal{R}_2)=0$, and we also know that  $\mathcal{R}(W_{u,T}^{k+1})\subseteq \mathcal{R}(W_{u,T}^k)$. To demonstarte the reverse containment, let $h\in \mathcal{R}(W_{u,T}^k)$. Then $h= W_{u,T}^k f$, $f\in L(p,q)$. Define $g:X\mapsto \mathbb{C}$ as   
	$$g(y)=\left\{\begin{array}{cl}
		\dfrac{f(x_y)}{u(y)} &\mbox{ if there is a unique } x_y\in T^{-1}(\mathcal{R}_1)\cap \bigg(T^k(X)\setminus \mathcal{F}_k\bigg): T_k(x_y)=y\\\\
		0 & \mbox{ otherwise }
	\end{array}\right..$$
	Clearly, $g$ is well defined and belongs to $L(p,q)$. Indeed, for $s\geq 0$, the distribution function of $g$ is given by 
	\begin{align*}
		\mu_g(s)&=\mu\bigg(\{y\in X~|~|g(y)|>s\}\bigg)\\
		&=\mu\bigg(\{y\in T^{k+1}(X)~|~\text{ we get } x_y\in T^{-1}(\mathcal{R}_1)\cap \bigg(T^k(X)\setminus \mathcal{F}_k\bigg) : T_k(x_y)=y \text{ and } \left|\frac{f(x_y)}{u(y)}\right|>s\}\bigg)\\
		&\leq \mu\bigg(\{T(x_y)\in T^{k+1}(X)\cap \mathcal{R}_1 ~\big|~ |f(x_y)|>s\epsilon\}\bigg)\\
		&\leq \mu\bigg(\{T(x_y)\in X~|~|f(x_y)|>s\epsilon\}\bigg)
		\leq \mu\bigg(\{x_y\in X~|~|f(x_y)|>s\epsilon\}\bigg)=\mu_f(s\epsilon).
	\end{align*}	
	Thus, $g^*(t)\leq \frac{1}{\epsilon}~ f^*(t)$ and $g^{**}(t)\leq \frac{1}{\epsilon}~ f^{**}(t)$, for each $t>0$. As a consequence, $\|g\|_{pq}\leq \frac{1}{\epsilon} \|f\|_{pq}$, $1<p\leq \infty$, $1\leq q\leq \infty$. Finally, the result follows as $f\in L(p,q)$. The target is to show that $g$ is the pre-image of $h$ under the operator $W_{u,T}^{k+1}$. For this purpose, we split the set $X$ as follows 
	\begin{align*}
		X&=\bigg( T^{-(k+1)}(\mathcal{R}_1)\cap T^{-k}(T^k(X)\setminus \mathcal{F}_k)\bigg)\cup \bigg(T^{-(k+1)}(\mathcal{R}_1)\cap T^{-k}(T^k(X)\setminus \mathcal{F}_k)\bigg)^c\\
		&=\bigg( T^{-(k+1)}(\mathcal{R}_1)\cap T^{-k}(T^k(X)\setminus \mathcal{F}_k)\bigg)\cup \bigg(T^{-(k+1)}(\mathcal{R}_2)\cup T^{-k}\bigg((X\setminus T^k(X)) \cup F_k\bigg)\bigg)
	\end{align*}
	Now, for every $x\in \bigg( T^{-(k+1)}(\mathcal{R}_1)\cap T^{-k}(T^k(X)\setminus \mathcal{F}_k)\bigg)$, we observe that 
	\begin{align*}
		\left(W_{u,T}^{k+1}g\right)(x) &= \prod\limits_{i=1}^{k+1} u\circ T^i(x)\cdot g\circ T^{k+1}(x)\\
		&= \prod\limits_{i=1}^{k} u\circ T^i(x)\cdot  u\circ T_k(T^{k}(x))\cdot  g\circ T_k(T^{k}(x))\\
		&= \prod\limits_{i=1}^{k} u\circ T^i(x)\cdot f\circ T^{k}(x)\\
		&= \left(W_{u,T}^k f\right)(x).
	\end{align*}
	As a consequence, 
	$$\{x\in X~|~\left(W_{u,T}^{k+1}g\right)(x)\neq \left(W_{u,T}^{k}f\right)(x)\}\subseteq \bigg(T^{-(k+1)}(\mathcal{R}_2)\cup T^{-k}\bigg((X\setminus T^k(X)) \cup F_k\bigg)\bigg),$$
	and as the superset has measure zero, $W_{u,T}^{k+1}g = W_{u,T}^{k}f = h$ a.e. on $X$. This implies that  $h$ is the member of $\mathcal{R}(W_{u,T}^{k+1})$. Thus, $\mathcal{R}(W_{u,T}^k)=\mathcal{R}(W_{u,T}^{k+1})$, that is, $\beta(W_{u,T})\leq k$. As stated, by contrapositive statement, we get that if the descent of $W_{u,T}$ is not finite, then there does not exist any non-negative integer $k$ such that $T_k = T|_{T^k(X)} : T^k(X)\rightarrow T^{k+1}(X)$ is injective almost everywhere on $T^k(X)$. Hence the result.
\end{proof}

	In the above theorem the condition $u$ is bounded away from zero cannot be relaxed which we justify with the help of following example. 
	
	\begin{Exa}
		Consider the Lebesgue measure space $(\mathbb{R},\mathcal{A},\mu)$. It is clearly $\sigma-$finite and complete. Define $T:\mathbb{R}\mapsto \mathbb{R}$ and $u:\mathbb{R}\mapsto \mathbb{C}$ as $T(x)=x-1$ and 
		\begin{align*}
			u(x) = \left\{\begin{array}{ll}
				0 & \mbox{ if } x\in [0,1]\\
				1 & \mbox{ otherwise }
			\end{array}\right.
		\end{align*} 
	respectively. With this pair of $(u,T)$, $W_{u,T}\in\mathfrak{B}(L(p,q))$, $1<p\leq \infty$, $1\leq q\leq \infty$ and $\mu\circ T(E) = \mu(E)$ for all measurable sets $E$. Moreover, the mappings $T_k : T^k(X)\rightarrow T^{k+1}(X)$ is injective almost everywhere on $T^k(X)$ for each $k\geq 0$. Finally, we show that $\chi_{[k+1,k+2]}\in\mathcal{R}(W_{u,T}^{k})\setminus \mathcal{R}(W_{u,T}^{k+1})$. The function $\chi_{[1,2]}$ is the pre-image of the function $\chi_{[k+1,k+2]}$ under the mapping $W_{u,T}^{k}$. Further, if we consider $h\in L(p,q)$ such that $\chi_{[k+1,k+2]}(x) = W_{u,T}^{k+1}h(x)$, then this leads to a contradiction as $u(x-(k+1))=0$. 
	\end{Exa}

The statement of the Theorem \ref{42} can be restated as follows:
	
	\begin{Thm}\label{417} 
		Suppose that $(X,\mathcal{A},\mu)$ is a $\sigma-$finite complete measure space, $T:X\rightarrow X$ satisfying $\mu(T(A))\leq \mu(A)$ for all $A\in\mathcal{A}$ and $u$ is bounded away from zero. If $k\in\mathbb{N}\cup\{0\}$ is the smallest number such that $T_k = T|_{T^k(X)} : T^k(X)\rightarrow T^{k+1}(X)$ is injective almost everywhere on $T^k(X)$ then the descent of $W_{u,T}$ is less than and equal to $k$.
	\end{Thm}

 	The reader can verify that the condition $u\neq 0$ a.e. on $X$ will give $h_T=0$ a.e. on $(Supp(u))^c$. In fact, 
\begin{align*}
	\mu(\{x\in (Supp(u))^c~|~h_T(x)\neq 0\})&=\mu(\{x\in X~|~h_T(x)\neq 0\}\cap \{x\in X~|~u(x)=0\})\\
	&\leq \mu(\{x\in X~|~u(x)=0\})=0.
\end{align*} 
We draw out here a conclusion that all the results that have been obtained with the assumption $h_T=0$ a.e. on $(Supp(u))^c$ do hold on replacing this condition by the condition $u\neq 0$ a.e. on $X$. 
Let us take another operator defined as ${\widehat{W}}_{u,T}(h)=u\cdot h\circ T$ on $L(p,q)$, $1<p\leq \infty$, $1\leq q\leq \infty$. If $u\in L_\infty(\mu)$ and $u\neq 0$ a.e. on $X$ then we can also deduce the following conclusions about $\widehat{W}_{u,T}$ without any extra efforts.
\begin{enumerate}[(a)]
	\item For every $k\in\mathbb{N}\cup\{0\}$, $\mathcal{N}\big(\widehat{W}_{u,T}^k\big)=L(p,q)(X_k)$.	
	\item A necessary and sufficient condition for $\alpha(\widehat{W}_{u,T})=k$ is that $k$ is the least non-negative integer such that $\mu_k$ and $\mu_{k+1}$ are equivalent measures.
	\item Suppose that $C_T$, $W_{u,T}$ and $\widehat{W}_{u,T}\in \mathfrak{B}(L(p,q))$, $1<p\leq \infty$, $1\leq q\leq \infty$. Then $\mathcal{N}(C_T^k)=\mathcal{N}(W_{u,T}^k)=\mathcal{N}\bigg(\widehat{W}_{u,T}^k\bigg)=L(p,q)(X_k)$ for every non-negative integer $k$. 
	\item If $C_T$, $W_{u,T}$ and $\widehat{W}_{u,T}$ are all belong to $\mathfrak{B}(L(p,q))$, $1<p\leq \infty$, $1\leq q\leq \infty$ then $\mu_k\ll\mu_{k+1}\ll\mu_k \Longleftrightarrow \mathcal{N}(C_T^k)=\mathcal{N}(C_T^{k+1})\Longleftrightarrow \mathcal{N}(W_{u,T}^k)=\mathcal{N}(W_{u,T}^{k+1})\Longleftrightarrow \mathcal{N}\bigg(\widehat{W}_{u,T}^{k}\bigg)=\mathcal{N}\bigg(\widehat{W}_{u,T}^{k+1}\bigg)$ for every $k\in\mathbb{N}\cup\{0\}$.
	\item If $C_T$, $W_{u,T}$, $\widehat{W}_{u,T}\in\mathfrak{B}(L(p,q))$, $1<p\leq \infty$, $1\leq q\leq \infty$ then $\alpha(C_T)=k$, $\alpha(W_{u,T})=k$, $\alpha(\widehat{W}_{u,T})=k$ and $k$ is the least non-negative integer satisfying $\mu_k\ll\mu_{k+1}\ll\mu_{k}$ are all equivalent statements.
	\item Let $\mu\circ T(A)\leq \mu(A)$ for all $A\in\mathcal{A}$. Then a sufficient condition for $\widehat{W}_{u,T}$ to have ascent $k$ is that $k$ is the smallest integer from the set $\mathbb{N}\cup\{0\}$ in such a way that $X_k^c\subseteq T^{k+1}(X_k^c)$, where $X_k^c= X\setminus X_k$.	
\end{enumerate} 	
Now, we justify that the condition $u\neq 0$ a.e. on X, that is, $\mu(\{x\in X~:~u(x)=0\})=0$ is essential for the results obtained above with the help of an example. 
\begin{Exa}
	On the Lebesgue measure space $([0,1],\mathcal{A},\mu)$, define $T$ and $u$ on $[0,1]$ as $T(x)=x$ and
	\begin{align*}
		u(x) = \left\{ \begin{array}{cl}
			0 & \mbox{ if } x\in [0,1/2)\\
			x & \mbox{ if }  x\in [1/2,1]
		\end{array}\right.
	\end{align*}  
respectively. Obviously, $\widehat{W}_{u,T}$ is a well defined weighted composition operator and $\chi_{[0,1/4]}\in\mathcal{N}(\widehat{W}_{u,T})\setminus \{0\}$. Also, $\mu_k=\mu$ and $X_k=\emptyset$ for every $k\in\mathbb{N}\cup\{0\}$. As a consequence, we get $\mu_0\ll\mu_1\ll\mu_0$, $X_0^c\subseteq T(X_0^c)$ and $\alpha(\widehat{W}_{u,T})\geq 1$. Hence, the statements of part (b) and (f) above are not true in this case.   
\end{Exa}
At the end, if we consider $N_0=X$ and $N_k=\{x\in X~|~u(T^{i}(x))\neq 0 \mbox{ for all } 0\leq i\leq k-1\}$ then Theorems \ref{38} and \ref{414} hold for $\widehat{W}_{u,T}$. Also, Theorems \ref{42} and \ref{417}, under the same conditions, are true for $\widehat{W}_{u,T}$.

	\section{Examples}
	
	The goal of this section is to give some applications which support the results achieved in the paper.

		\begin{Exa}
		On the measure space $(\mathbb{N},2^{\mathbb{N}},\mu)$, define $T:\mathbb{N}\rightarrow \mathbb{N}$ as $T(m)=2n$ if $m\in\{2n-1,2n\}_{n\in \mathbb{N}}$ and $u:\mathbb{N}\rightarrow \mathbb{C}$ as $u(n)=\left\{\begin{array}{cl}
			0 & \mbox{ if $n$ is odd}\\
			1 & \mbox{ if $n$ is even} 
		\end{array}\right.$. From the definition of $u$ and $T$, we can conclude that $(Supp(u))^c=\{2n-1:n\in\mathbb{N}\}$ and $h_T$ is given by
$$h_T(n)=\left\{\begin{array}{cl}
			0 & \mbox{if $n$ is odd}\\
			2 & \mbox{if $n$ is even}
		\end{array}.\right.$$ 
Clearly,  $h_T=0$ on $(Supp(u))^c$. Further, for every $k\in\mathbb{N}$, $T^k=T$, $h_{T^k}=h_T$ and $X_k^c\subseteq T^{k+1}(X_k^c)$. By Theorem \ref{411}, the ascent of $W_{u,T}$ is equal to $1$. 
	\end{Exa}

\begin{Exa}
	Consider the $\sigma-$finite complete measure space $\left([0,1], \mathcal{A}, \mu\right)$ with $\mu$ the Lebesgue measure. Define $T:[0,1]\mapsto [0,1]$ and $u:[0,1]\mapsto \mathbb{C}$, respectively, as $T(x)=\frac{x}{2}$ and 
	\begin{align*}
			u(x) = \left\{\begin{array}{ll}
			x & x \mbox{ is irrational in } [0,1]\\
			0   & x \mbox{ is rational in } [0,1]
		\end{array}.\right.
	\end{align*} 
With this pair of mappings, we have $W_{u,T}\in \mathfrak{B}(L(p,q))$, $1<p\leq \infty$, $1\leq q\leq \infty$. Take $A_k=\left(\left(\frac{1}{2}\right)^{k+1}, \left(\frac{1}{2}\right)^k\right)$. These choices give $N_k=\{x\in [0,1]:x \mbox{ is irrational}\}$ and $T^{-k}(A_k)= \left(\frac{1}{2},1\right)$. As a consequence, we have $\mu(T^{-k}(A_k)\cap N_k)=\frac{1}{2}\neq 0$ and $\mu(T^{-(k+1)}(A_k)\cap N_{k+1})=\mu(\emptyset)=0$. By Theorem \ref{38}, $\alpha\left(W_{u,T}\right) = \infty$.
\end{Exa}

	\begin{Exa}
	Consider the $\sigma-$finite and complete measure space $([0,\infty),\mathcal{A},\mu)$, where $\mu$ is the Lebesgue measure.  Define $T:[0,\infty)\mapsto [0,\infty)$ and $u:[0,\infty)\mapsto \mathbb{C}$ by setting 
	\begin{align*}
		T(x) = \left\{\begin{array}{ll}
			x & \mbox{ if } x\in [0,1]\\
			x-1 & \mbox{ elsewhere }
		\end{array}\right.
	\end{align*} 
	 and 
	$u(x)=1$ respectively. With this pair of $(u,T)$, $W_{u,T}\in \mathfrak{B}(L(p,q))$, $1<p\leq \infty$, $1\leq q\leq \infty$, and $\mu\circ T(A)\leq \mu(A)$ for all $A\in\mathcal{A}$. Also,  $N_{k+1}=[0,\infty)$, $A_1^k=(0,1]$, $A_2^k=(1,2]$ and $T_k(A_1^k)=T_k(A_2^k)$ where $T_k = T|_{T^k(N_{k+1})}: T^k(N_{k+1})\rightarrow T^{k+1}(N_{k+1})$ for every $k\in\mathbb{N}\cup\{0\}$. Thus, using Theorem \ref{new419}, we get $\beta(W_{u,T})=\infty$. 
\end{Exa}

\begin{Exa}
		Let $([-1,1],\mathcal{A},\mu)$ be a $\sigma-$finite complete Lebesgue measure space. Define $u:[-1,1]\rightarrow \mathbb{C}$ as $u\equiv 1$ and $T:[-1,1]\rightarrow [-1,1]$ as 
		$$T(x)=\left\{\begin{array}{cl}
			\frac{x}{2} & \mbox{ if } 0\leq x\leq 1\\\\
			\frac{-x}{2} & \mbox{ if } -1\leq x\leq 0
		\end{array}\right..$$
		Now, for any measurable set $A\in\mathcal{A}$, we have $T^{-1}(A)=2\left(A\cap [0,\frac{1}{2}]\right)\cup (-2)\left(A\cap [0,\frac{1}{2}]\right)$ and $T(A)=\frac{A_1}{2}\cup \frac{-A_2}{2}$ where $A_1=A\cap [0,1]$ and $A_2=A\cap [-1,0]$. This implies that $\mu(T^{-1}(A))\leq 2 \mu\left(A\cap [0,\frac{1}{2}]\right)+2
		\mu\left(A\cap [0,\frac{1}{2}]\right)\leq 4 \mu(A)$ and $\mu(T(A))\leq \mu(A)$. As a conclusion $h_T\in L_{\infty}(\mu)$, and hence $W_{u,T}\in \mathfrak{B}(L(p,q))$. 
		Finally, note that, for every natural number $k$, $T_k=T|_{T^k([-1,1])}:T^k([-1,1])\rightarrow T^{k+1}([-1,1])$ is injective almost everywhere on $T^k([-1,1])$. By Theorem \ref{417}, we have descent of $W_{u,T}$ is less than and equal to $1$. Moreover, we can see that $\beta(W_{u,T})=1$.	For this purpose one can verify that the non-zero characteristic function $\chi_{[-1,0]}$ from $L(p,q)$ is not the member of $\mathcal{R}(W_{u,T})$, i.e., it has no preimage under $W_{u,T}$. In fact, suppose that $W_{u,T}h=\chi_{[-1,0]}$ for some $h\in L(p,q)$. Then $\mu(E_1^c)=\mu(E_2^c)=0$ where $E_1:=\{x\in [-1,0)~|~h\circ T(x)=1\}$, $E_2:=\{x\in (0,1]~|~h\circ T(x)=0\}$, $E_1^c:=[-1,0)\setminus E_1$ and $E_2^c:=(0,1]\setminus E_2$. Finally, the result follows using the fact that $T$ is non-singular, $E_1\subseteq T^{-1}(T(E_1))$ and $\mu(T(E_1))\leq \mu(T(E_2^c))\leq \mu(E_2^c)=0$. 
\end{Exa}

\section*{Statements and Declarations}
 No funding was received for conducting this study. Authors have no competing interest to declare that are relevant to the content of this article.


\begin{thebibliography}{0000}

\bibitem{adv} S.C. Arora, G. Datt, and S. Verma, Weighted composition operators on Lorentz spaces, Bull. Korean Math. Soc., 44 (4) (2007), 701–708.
		
\bibitem{bg} D.S. Bajaj and G. Datt, Ascent and descent of composition operators on Lorentz spaces, Commun. Korean Math. Soc., 37 (1) (2022), 195–205.

\bibitem{cc1} R.E. Castillo and H.C. Chaparro, Classical and Multidimensional Lorentz Spaces, Walter De Gruyter, 2021.

\bibitem{cc2} R.E. Castillo and H.C. Chaparro, Weighted composition operator on two-dimensional Lorentz spaces, Math. Inequal. Appl., 20 (2017), 773-799.

\bibitem{cwikel} M. Cwikel, The dual of weak $L^p$, Annales De L’institut Fourier, 25 (2) (1975), 81–126.

\bibitem{cf1} M. Cwikel and C. Fefferman, Maximal seminorms on Weak $L^{1}$, Stud. Math., 69 (1981), 149–154.

\bibitem{cf2}	M. Cwikel and C. Fefferman, The canonical seminorm on Weak $L^1$, Stud. Math., 3 (78) (1984), 275–278.

\bibitem{db} G. Datt and D.S. Bajaj, Essential ascent and descent of weighted composition operators on Lorentz sequence spaces, J. Math. Anal. Appl., 514 (2) (2022), 126333.

\bibitem{emr}	Y. Estaremi, S. Maghsoudi, and I. Rahmani, A note on multiplication and composition operators in Orlicz spaces, Filomat, 32 (7) (2018), 2693–2699.

\bibitem{gks} S. Gupta, B.S. Komal, and N. Suri, Weighted composition operators on Orlicz spaces, Int. J. Contemp. Math. Sci., 5(1) (2010), 11-20. 

\bibitem{hunt} R.A. Hunt, On $L(p, q)$ spaces, Enseignment Math., 12 (2) (1966), 249–276.

\bibitem{kumar} R. Kumar, Weighted composition operators between two $L^p-$spaces, Math. Vesn., 61 (236) (2009), 111-118.

\bibitem{lor1} G.G. Lorentz, On the theory of spaces $\Lambda$, Pac. J. Math., 1 (3) (1951), 411–429.

\bibitem{lor2} G.G. Lorentz, Some New Functional Spaces, Ann. Math., 51 (1) (1950), 37–55.

\bibitem{pkjf} L. Pick, A. Kufner, O. John, and S. Fu\v{c}ík, Function Spaces, vol. 1 ($2^{\mbox{nd}}$ Edition), Walter de Gruyter, 2012.

\bibitem{tk} Takagi, Hiroyuki, and Katsuhiko Yokouchi, Multiplication and Composition Operators. In Function Spaces: Proceedings of the Third Conference on Function Spaces, May 19-23, 1998, Southern Illinois University at Edwardsville, vol. 232, p. 321. American Mathematical Soc., 1999.

\bibitem{tay1} A.E. Taylor, Introduction to Functional Analysis, John Wiley and Sons, 1958.

\bibitem{tay2} A.E. Taylor, Theorems on ascent, descent, nullity and defect of linear operators, Math. Ann., 163(1) (1966), 18-49.
		
	\end{thebibliography}
\end{document}